\documentclass[a4paper,11pt,fleqn]{amsart}
\usepackage{mathtools,xfrac}

\usepackage{hyperref} 

\usepackage{amssymb,amsmath,amsthm}
\usepackage{graphicx}
\usepackage{color}
\usepackage{thmtools}
\usepackage{thm-restate}
\usepackage{ulem}
\usepackage[shortlabels]{enumitem}
\usepackage{mathrsfs}
\usepackage[margin=1in]{geometry}
\usepackage{comment}
\usepackage[utf8]{inputenc}
\usepackage[T1]{fontenc}

\newtheorem{theorem}{Theorem}[section]
\newtheorem{lemma}[theorem]{Lemma}
\newtheorem{proposition}[theorem]{Proposition}
\newtheorem{problem}[theorem]{Problem}

\newtheorem{corollary}[theorem]{Corollary}

\theoremstyle{definition}
\newtheorem{definition}[theorem]{Definition}

\theoremstyle{remark}
\newtheorem{remark}[theorem]{Remark}
\newtheorem{notation}[theorem]{Notation}

\numberwithin{equation}{section}

\newcommand{\N}{\mathbb{N}}
\newcommand{\R}{\mathbb{R}}

\newcommand{\bB}{\mathcal{B}}

\newcommand{\lL}{\mathcal{L}}

\newcommand{\sS}{\mathcal{S}}

\newcommand{\xX}{\mathcal{X}}

\DeclareMathOperator{\supp}{supp}
\DeclareMathOperator{\range}{range}
\DeclareMathOperator{\spn}{span}
\DeclareMathOperator{\ord}{ord}
\DeclareMathOperator{\clspn}{\overline{\mathrm{span}}}%\newcommand{\clspn}{{\overline{\rm span}\,}}

\newcommand{\norm}[1]{\left\lVert#1\right\rVert}
\newcommand\restrict[1]{\raisebox{-.5ex}{$|$}_{#1}}

\author[A. Pelczar-Barwacz]{Anna Pelczar-Barwacz}
\address{Faculty of Mathematics and Computer Science, Jagiellonian University, {\L}ojasiewicza 6, 30-348 Krak\'ow, Poland}
\email{anna.pelczar@uj.edu.pl}

\author[Z. Silber]{Zden\v{e}k Silber}
\address{Institute of Mathematics of the Czech Academy of Sciences,
Žitná 25, 115 67 Prague 1, Czech Republic}
\email{zdesil@seznam.cz}

\author[T. Wawrzycki]{Tomasz Wawrzycki}
\address{Faculty of Mathematics and Computer Science, Jagiellonian University, {\L}ojasiewicza 6, 30-348 Krak\'ow, Poland}
\email{tomasz.wawrzycki@student.uj.edu.pl}

\keywords{Intrinsic Baire classes, universal Banach space, Baire order}

\subjclass[2020]{46B03, 46B10, 46B20, 54H05}

\thanks{The research of the second named author was supported by the project L100192501 funded by the "Programme to support prospective human resources – post Ph.D. candidates" of the Czech Academy of Sciences and RVO 67985840. The research of the third named author was funded by the program Excellence Initiative – Research University at the Jagiellonian
University in Kraków. The research presented in this paper was partially conducted during the authors’ visit at the Erwin Schr\"odinger International Institute for Mathematics and Physics
in Vienna, participating in the workshop \textit{Structures in Banach spaces}.}

\begin{document}

\title{Banach spaces with arbitrary finite Baire order}

\begin{abstract}
    We investigate intrinsic Baire classes of Banach spaces defined by Argyros, Godefroy and Rosenthal in \cite{ArgyrosGodefroyRosenthal2003}. We introduce a construction, for any Banach space $X$ with a basis, of an $\ell_1$-saturated separable Banach space $Y$ such that for any $\alpha \leqslant \omega_1$ we have $Y^{**}_{1+\alpha} \cong Y \oplus X^{**}_\alpha$, where $X^{**}_\alpha$ denotes the $\alpha$-th intrinsic Baire class of $X$. We apply this construction to answer two open problems from \cite{ArgyrosGodefroyRosenthal2003},  namely we build separable Banach spaces of any Baire order less or equal to $\omega$, and a non-universal separable Banach space of order $\omega_1$. Finally, we apply the construction to show an analogue of a result of Lindenstrauss \cite{Lindenstrauss1971} by constructing, for any Banach space $X$ with a basis and any $n \in \N$, a Banach space $Y$ such that $Y^{**}_n \cong Y^{**}_{n-1} \oplus X$, showing that any such $X$ can appear as the space of functionals in a bidual Banach space $Y^{**}$ that are of $n$-th intrinsic Baire class but not of $(n-1)$-th intrinsic Baire class.
\end{abstract}

\maketitle

\section{Introduction}

Let $X$ be a Banach space. Following Argyros, Godefroy and Rosenthal \cite{ArgyrosGodefroyRosenthal2003}, we define the \textbf{Baire classes of $X$} in the following recursive way: we set $X^{**}_0 = X \subset X^{**}$, and for $\alpha \leqslant \omega_1$ we set $X^{**}_\alpha$ to be the set of all limits of weak* convergent sequences from $\bigcup_{\beta < \alpha}X^{**}_\beta$.

When $X$ is separable (which is the case we focus on in this paper), the dual ball $K = (B_{X^*},w^*)$ is a metrizable compact space. We define, for $\alpha \leqslant \omega_1$, the sets $X^{**}_{\bB_\alpha} = \{x^{**} \in X^{**}: x^{**}|_K \text{ is a Baire-$\alpha$ function}\}$. It is important to distinguish between $X^{**}_\alpha$ and $X^{**}_{\bB_\alpha}$ -- while clearly any element of $X^{**}_\alpha$ is a Baire-$\alpha$ function on $K$, and thus an element of $X^{**}_{\bB_\alpha}$, the opposite inclusion is not valid in general, as demonstrated by Talagrand \cite{Talagrand1984}. However, for $\alpha \in \{0,1\}$ we have the equality $X^{**}_\alpha = X^{**}_{\bB_\alpha}$ by the Banach-Dieudonn\'e theorem for $\alpha = 0$ and by the results of Choquet and Mokobodski (see \cite[Theorem II.1.2(a)]{ArgyrosGodefroyRosenthal2003}) for $\alpha = 1$. In particular, $X^{**}_0$ and $X^{**}_1$ are norm closed in $X^{**}$. We would like to note that in \cite[page 1047]{ArgyrosGodefroyRosenthal2003} it is claimed that $X^{**}_\alpha = X^{**}_{\bB_\alpha}$ for every $\alpha \leqslant \omega_1$ in the case when $X$ is an $\lL_\infty$ space. The argument, however, contains a gap and the result is not true as shown in \cite[Theorem 1.2]{Spurny2010}. Note that in \cite{ArgyrosGodefroyRosenthal2003} the classes $(X_\alpha^{**})_{\alpha \leqslant \omega_1}$ are called the intrinsic Baire classes, but as we will not consider the classes $(X_{\bB_\alpha}^{**})_{\alpha \leqslant \omega_1}$ further in the paper, we use the name Baire classes of $X$ for $(X_\alpha^{**})_{\alpha \leqslant \omega_1}$ for the sake of brevity as it will cause no confusion.

Recall that a Banach space $X$ is \textbf{weakly sequentially complete} if every weak Cauchy sequence in $X$ (that is a sequence $(x_n)_{n=1}^\infty\subset  X$ with the scalar sequence $ (x^*(x_n))_{n-1}^\infty$ convergent for any $x^* \in X^*$) is weakly convergent. 
It follows from the uniform boundedness principle that a sequence in $X$ is weak Cauchy if and only if it is weak* convergent in $X^{**}$, and is weakly convergent if and only if its weak* limit is an element of $X$. Hence, $X$ is weakly sequentially complete if and only if it is weak* sequentially closed in $X^{**}$, or in our notation, if and only if $X^{**}_0 = X^{**}_1$. By the Odell-Rosenthal theorem \cite{OdellRosenthal1975}, a separable Banach space $X$ does not contain $\ell_1$ if and only if every element of the bidual $X^{**}$ can be attained as a weak* limit of a sequence from $X$, or in our notation, if and only if $X^{**}_1 = X^{**}$. In the general setting (of not necessarily separable Banach spaces) the Odell-Rosenthal theorem yields the following result: 
a Banach space $X$ not containing $\ell_1$ satisfies $X_1^{**}=X_{\omega_1}^{**}$. 
%If $X$ is nonseparable, we still get that if $X$ does not contain $\ell_1$, then $X_1^{**} = X_{\omega_1}^{**}$. 
Indeed, for any $x^{**} \in X_{\omega_1}^{**}$, there is a countable subset $C$ of $X$ such that $x^{**} \in \overline{C}^{w^*}$. Taking $Y$ to be the closed linear span of $C$ and canonically identifying $Y^{**}$ with $\overline{Y}^{w^*} \subset X^{**}$, we obtain $x^{**} \in Y^{**}$. As $Y$ is separable, by the Odell-Rosenthal theorem  $x^{**} \in Y_1^{**} \subset X_1^{**}$.

Let us now illustrate the possible structure of Baire classes  on examples:
\begin{enumerate}[(i)]
    \item If $X$ is weakly sequentially complete, then $X^{**}_\alpha = X$ for any $\alpha \leqslant \omega_1$;
    \item If $X$ is not reflexive  separable
    and does not contain $\ell_1$ (e.g. $X = c_0$), then $X^{**}_0 \subsetneq X^{**}_\alpha = X^{**}$ for any $1 \leqslant \alpha \leqslant \omega_1$. 
    \item If $X = \mathcal{C}([0,1])$, then $X^{**}_\alpha \subsetneq X^{**}_\beta$ for any $0 \leqslant \alpha < \beta \leqslant \omega_1$.
\end{enumerate}
Indeed, the first two points were explained above and the third point follows as for $\alpha \leqslant \omega_1$ the space $\mathcal{C}([0,1])^{**}_\alpha$ can be identified with the space $\bB_\alpha([0,1])$ of bounded Baire-$\alpha$ functions on $[0,1]$ by the dominated convergence theorem, and by a well known result from descriptive set theory,  $\bB_\alpha([0,1])$, $\alpha \leqslant \omega_1$, form a strictly increasing chain. This motivates the following definition that will be the core of our paper:

\begin{definition}
    Let $X$ be a Banach space and $\alpha \leqslant \omega_1$. We say that $X$ has the \textbf{Baire order $\alpha$}, and write $\ord_\bB(X) = \alpha$, if $\alpha$ is the smallest ordinal such that $X^{**}_\alpha = X^{**}_{\alpha+1}$.
\end{definition}
Note here that 
any Banach space is of the Baire order $\alpha$, for some $\alpha\leqslant\omega_1$. 
The examples above provide separable Banach spaces of orders $0,1$ and $\omega_1$. It was unknown whether these three values are the only possible orders of separable Banach spaces. In \cite{ArgyrosGodefroyRosenthal2003}, Argyros, Godefroy and Rosenthal asked the following:

\begin{problem}[\protect{\cite[Problem 5]{ArgyrosGodefroyRosenthal2003}}] \label{Problem-NontrivialOrder}
    Let $\alpha \leqslant \omega_1$, $\alpha \notin \{0,1,\omega_1\}$. Does there exist a separable Banach space of the Baire order $\alpha$?
\end{problem}

We give a positive answer for any $\alpha \leqslant \omega$ in Theorem \ref{Theorem-MainOrder}, building a Banach space $Y(\alpha)$ of the Baire order $\alpha$, which is additionally $\ell_1$-saturated (recall that $X$ is \textbf{$\ell_1$-saturated} if every closed infinite-dimensional subspace of $X$ contains an isomorphic copy of $\ell_1$).

%Another open problem  concerned  the relation of $X$ having the Baire order $\omega_1$ and $X$ being universal in the class of separable Banach spaces (recall that $X$ is \textbf{universal} if it contains an isomorphic copy of every separable Banach space). 
Another open problem  concerned  the relation between the Baire order $\omega_1$ and universality in the class of separable Banach spaces (recall that $X$ is \textbf{universal} for separable Banach spaces if it contains an isomorphic copy of every separable Banach space). On one hand, every universal space contains a complemented copy of $\mathcal{C}([0,1])$, and thus has the Baire order $\omega_1$ (see \cite[page 1044]{ArgyrosGodefroyRosenthal2003}). The validity of the opposite implication was the subject of another question of Argyros, Godefroy and Rosenthal:

\begin{problem}[\protect{\cite[Problem 4]{ArgyrosGodefroyRosenthal2003}}] \label{Problem-Universal}
    Let $X$ be a separable Banach space of the Baire order $\omega_1$. Is $X$ universal?
\end{problem}
We give a negative answer  in Theorem \ref{Theorem-MainUniversal} by constructing an $\ell_1$-saturated separable Banach space of the Baire order $\omega_1$. Clearly, an $\ell_1$-saturated Banach space cannot be universal.

Both of the answers to Problems \ref{Problem-NontrivialOrder} and \ref{Problem-Universal} will be deduced from following theorem, which will be the core of our paper.

\begin{theorem} \label{Theorem-Construction}
    For every Banach space $X$ with a basis there is an $\ell_1$-saturated Banach space $Y$ with a boundedly complete basis such that for any $\alpha\leqslant\omega_1$,
    \begin{align*}
        Y_{1+\alpha}^{**}\cong Y\oplus X_{\alpha}^{**}.
    \end{align*}
    More precisely, there is a weak*-to-weak* continuous isomorphic embedding $L:X^{**}\rightarrow Y^{**}$ such that $Y_{1+\alpha}^{**} = \iota(Y)\oplus L(X_{\alpha}^{**})$ for any  $\alpha\leqslant\omega_1$, where $\iota:Y\hookrightarrow Y^{**}$ is the canonical embedding.

    In particular, $\ord_\bB(Y) = 1 + \ord_\bB(X)$.
\end{theorem}

Concerning Problem \ref{Problem-NontrivialOrder}, we apply Theorem \ref{Theorem-Construction} recursively for $X$ with $\ord_{\bB}(X) = n$ to obtain $Y$ with $\ord_\bB (Y) = n+1$, which yields the result for finite ordinals. We then use an $\ell_1$-sum argument to establish the result for the ordinal $\omega$. Applying Theorem \ref{Theorem-Construction} for $X = \mathcal{C}([0,1])$ yields negative answer to Problem \ref{Problem-Universal}. We note that the main drawback of our approach to give the complete answer to Problem \ref{Problem-NontrivialOrder} is that we can build, from a separable Banach space $X$ with $\ord_\bB(X) = \alpha$, a separable Banach space $Y$ with $\ord_\bB(Y) = 1 + \alpha$, instead of $\ord_\bB(Y) = \alpha + 1$ needed for the recursive construction for infinite ordinals. Finally, in Theorem \ref{Theorem-LindenstraussAnalogue}, we obtain, in the class of Banach spaces with a basis, a Baire classes analogue of a well known theorem of Lindenstrauss \cite{Lindenstrauss1971} (see also \cite{James1960}), stating that for any separable Banach space $X$, there is a separable Banach space $\xX$ with $\xX^{**} / \xX$ isomorphic to $X$ (generalized in \cite{DavisEtAl1974} to weakly compactly generated Banach spaces).

A key building stone of our construction will be the space $X_{AH}$ constructed by Azimi and Hagler in \cite{AzimiHagler1986}. This space was constructed as an example of an $\ell_1$-saturated Banach space which fails the Schur property. What is important in our context is that $X_{AH}$ is of codimension one in its first Baire class $(X_{AH})^{**}_1$, that is $(X_{AH})^{**}_1 = X_{AH} \oplus \R$. The space $X_{AH}$ can be thus understood as a Baire-1 variant of the James space \cite{James1951}, which is of codimension one in its whole bidual. The fact that the difference between $X_{AH}$ and $(X_{AH})^{**}_1$ is as small as possible will allow us to control the structure of Baire classes in our construction. The precise definition of $X_{AH}$ and some of its properties will be recalled in Section \ref{section_AH_space}.

Before proceeding with the construction let us mention that the study of Baire classes is closely related to the study of affine classes of functions on compact convex sets (see e.g. \cite{Spurny2010}), and of weak* sequential closures and weak* derived sets (see e.g. the survey paper \cite{Ostrovskii2002}). Finally, we note that Problem \ref{Problem-NontrivialOrder} has negative solution if we restrict our attention only to the class of Lindenstrauss spaces (recall that $X$ is a \textbf{Lindenstrauss space} if its dual is isometric to an $L_1(\mu)$ space). It follows from the results of \cite{LudvikSpurny2014}, but not explicitly stated in \cite{LudvikSpurny2014}, therefore for the sake of completeness we give an  argument in Section \ref{section_lindenstrauss}.

Finally we present the following question concerning nonseparable setting, posed by Ond\v{r}ej Kalenda.
\begin{problem}
Let $X$ be a Banach space of the Baire order $\omega_1$. Does $X$ admit a separable subspace of the Baire order $\omega_1$?
\end{problem}

The paper is organized as follows. At the end of this introductory section we recall the basic notions and some of their properties. Section 2. contains general observations on Baire classes of a Banach space, as well as description of the Baire order of separable Lindenstrauss spaces. Section 3. is devoted to the proof of Theorem \ref{Theorem-Construction}. We gather its consequences, answering the problems mentioned above, in Section 4. 

\subsection*{Acknowledgements}
We are deeply indebted to Ond\v{r}ej Kalenda for reading the first version of the manuscript and valuable comments. 

\subsection*{Basic definitions and notation}
\label{basics}

We shall use the standard terminology and notation, for any unexplained notation or concept we refer to \cite{AlbiacKalton,FabianEtAl,Kechris}. 

Let $X$ be a Banach space. By $B_X$ we denote the closed unit ball of $X$. A set $B\subset B_{X}$ is $c$-\textbf{norming}, $c\geqslant 1$, for  $Z\subset X^*$, provided $\|x^*\|\leqslant c\sup\limits_{x\in B}|x^*(x)|$ for every  $x^*\in Z$.  A set $B\subset B_X$ is called norming for $Z\subset X^*$, if it is $c$-norming for some $c\geqslant 1$. 

Recall that any dual Banach space $X$, i.e. a space of the form $X=Z^*$ for some Banach space $Z$, is complemented in its bidual space $X^{**}$ by means of the \textbf{Dixmier projection} $P=\kappa^*:X^{**}\ni x^{**}\mapsto x^{**}\restrict{\kappa(Z)}\in X=Z^*$, where $\kappa:Z\hookrightarrow X^*$ is the canonical embedding of $Z$ into its bidual space $Z^{**}=X^*$. 

Fix Banach spaces $X,Y$. We write $X\cong Y$ if $X,Y$ are isomorphic. We say that a basic sequence $(x_n)_{n=1}^\infty\subset X$ \textbf{dominates} another basic sequence $(y_n)_{n=1}^\infty\subset Y$ provided the mapping $T$ carrying each $x_n$ to $y_n$, $n\in\N$, extends to a bounded operator $T:\spn\{x_n:n\in\N\}\to\spn\{y_n:n\in\N\}$. We say that basic sequences $(x_n)_{n=1}^\infty\subset X$ and $(y_n)_{n=1}^\infty\subset Y$ are \textbf{equivalent} if $(x_n)_{n=1}^\infty$ dominates $(y_n)_{n=1}^\infty$ and vice versa. 

Recall that a (Schauder) basis $(e_n)_{n=1}^\infty$ of $X$ is \textbf{boundedly complete}, if for any scalars $(a_n)_{n=1}^\infty$, the series $\sum_{n=1}^\infty a_ne_n$ converges (in norm) whenever  $\sup_N\|\sum_{n=1}^N a_ne_n\|<\infty$ and is \textbf{bimonotone} if $\norm{P_m - P_n} = 1$ for every $n < m$, where $(P_k)_{k=1}^\infty$ are the canonical projections associated to the basis $(e_n)_{n = 1}^\infty$.

A formal series $\sum_{n=1}^\infty x_n$ of vectors $(x_n)_{n=1}^\infty\subset X$ is \textbf{weakly unconditionally Cauchy} whenever $\sum_{n=1}^\infty |x^*(x_n)|<\infty$ for every $x^*\in X^*$, equivalently (see \cite[Lemma 2.4.6]{AlbiacKalton}) if 
\[\sup\left\{\left\|\sum_{n=1}^N\varepsilon_nx_n\right\|:(\varepsilon_n)_n\subset\{-1,1\}^N, N\in\N\right\}<\infty.\]
Moreover, a formal series $\sum_{n=1}^\infty x_n^*$ of functionals $(x_n^*)_{n=1}^\infty\subset X^*$ with $\sum_{n=1}^\infty |x_n^*(x)|<\infty$ for every $x \in X$,  is weakly unconditionally Cauchy (use the Banach-Steinhaus theorem). 

Given normed spaces $(X_n,\|\cdot\|_n)$, $n\in\N$, we define their $\ell_1$-direct sum by the  formula 
\[\big(\bigoplus X_n\big)_{\ell_1}= \Big\{x=(x_n)_{n=1}^\infty\in\prod_{n=1}^\infty X_n: \|x\|:=\sum_{n=1}^\infty \|x_n\|_n<\infty\Big\}.\]
It follows easily that $\big((\bigoplus X_n)_{\ell_1}, \|\cdot\|\big)$ is a Banach space if and only if $X_n$ is a Banach space for every $n \in \N$. We shall need the following simple fact, we include the proof for the sake of completeness. Recall that $X$ is \textbf{$\ell_1$-saturated}, provided any its closed infinite dimensional subspace contains an isomorphic copy of $\ell_1$. 

\begin{lemma}\label{Lemma_ell_1_saturation}
    Let $X_n$, $n\in\N$, be $\ell_1$-saturated Banach spaces. Then $(\bigoplus X_n)_{\ell_1}$ is also $\ell_1$-saturated. 
\end{lemma}

\begin{proof}
    Let $Z$ be a closed infinite dimensional subspace of $(\bigoplus X_n)_{\ell_1}$. Denote by $\pi_n: (\bigoplus X_n)_{\ell_1}\to X_n$, $n\in\N$, the canonical projections. If there is  $n\in\N$ and $c>0$ with $\|\pi_nz\|\geqslant c\|z\|$ for any $z\in (\bigoplus X_n)_{\ell_1}$, then $\pi_n\restrict{Z}$ is an isomorphic embedding of $Z$ into the $\ell_1$-saturated space $X_n$, so $Z$ is $\ell_1$-saturated. Otherwise, using the definition of the norm on $(\bigoplus X_n)_{\ell_1}$, we pick inductively a normalized sequence $(z_n)_n\subset Z$ and $(k_n)_n\subset\N$ so that $\|z_n-\tilde z_n\|<2^{-n}$, where $\tilde z_n=(\pi_{k_{n+1}}-\pi_{k_n})(z_n)$,  $n\in\N$. By the definition of the norm in $Z$ the sequence $(\tilde z_n)_n$ is equivalent to the unit vector basis of $\ell_1$, whereas by the Krein-Milman-Rutman theorem (see \cite[Theorem 1.3.9]{AlbiacKalton}) $(z_n)_{n\in J}$ and $(\tilde z_n)_{n\in J}$ are equivalent for some infinite $J\subset\N$.
\end{proof}

\section{Baire functionals}

 We gather now observations on Baire classes of general Banach spaces, needed in the proof of Theorem \ref{Theorem-Construction}, and discuss the Lindenstrauss spaces case.

\subsection{General facts}\label{Section_general}

We recall that for a Banach space $X$, the \textbf{Baire classes} $X^{**}_\alpha$, $\alpha\leqslant\omega_1$ are defined recursively as follows. We set $X^{**}_0 = X \subset X^{**}$ (identifying $X$ with its canonical embedding in $X^{**}$)  and, if $X^{**}_\beta\subset X^{**}$, $\beta < \alpha$ are already defined, we set $X^{**}_\alpha\subset X^{**}$ to be the set of all limits of weak* convergent sequences from $\bigcup_{\beta < \alpha}X^{**}_\beta$. Elements of each Baire class $X^{**}_\alpha$ will be called \textbf{Baire-$\alpha$ functionals} of $X$. Note that in \cite{ArgyrosGodefroyRosenthal2003} the classes $(X_\alpha^{**})_{\alpha \leqslant\omega_1}$ are called the intrinsic Baire classes of $X$. We chose to drop the word "intrinsic" for the sake of brevity. 

The following Lemma may be concluded using \cite[Lemma 3.3.]{Smith}, but we include an elementary proof for the sake of completeness.

\begin{lemma}\label{norming-set}
    Let $\mathcal{S}\subset B_{X^*}$ be a weak* compact $c$-norming set for a Banach space $X$, $c\geqslant 1$. Then 
    $\mathcal{S}$ is $c$-norming for $X_{\omega_1}^{**}$. Moreover, a sequence $(x_n^{**})_n\subset X_{\omega_1}^{**}$ is weak* convergent  if and only if it is bounded and pointwise convergent on $\mathcal{S}$.
\end{lemma}
\begin{proof}
    Consider 
    \[T:X\ni x\mapsto x\restrict{\mathcal{S}}\in\mathcal{C}(\mathcal{S}),\]
    where $\mathcal{C}(\mathcal{S})$ denotes the space of continuous functions on $\mathcal{S}$ equipped with a weak* topology. We will show inductively that $T^{**}(X_\alpha^{**})\subset\mathcal{B}_\alpha(\mathcal{S})$, the space of bounded Baire-$\alpha$ functions on $\mathcal{S}$, for every $\alpha\leqslant\omega_1$, and $T^{**}(x^{**})=x^{**}\restrict{\mathcal{S}}$ for every $x^{**}\in X_{\omega_1}^{**}$. The case $\alpha=0$ is clear. Now, for a sequence $(x_n^{**})_n\in\bigcup_{\beta<\alpha}X_{\beta}^{**}$ which is weak* convergent to $x^{**}\in X_\alpha^{**}$, we have
    \[T^{**}(x^{**})(\mu)=\lim_{n\rightarrow\infty}T^{**}(x_n^{**})(\mu), \ \mu\in\mathcal{M}(\mathcal{S})=\mathcal{C}(\mathcal{S})^*.\]
    By the inductive hypothesis $T^{**}(x_n^{**})\in\bigcup_{\beta<\alpha}\mathcal{B}_{\beta}(\mathcal{S})$ for every $n \in \N$, hence
    \[T^{**}(x^{**}_n)(\mu)=\int_\mathcal{S}T^{**}(x^{**}_n)\,d\mu.\]
    By the Lebesgue dominated convergence theorem,  $T^{**}(x^{**})=\lim\limits_{n\rightarrow\infty}T^{**}(x_n^{**})$, where the limit on the right-hand side is pointwise on $\mathcal{S}$, thus $T^{**}(x^{**})\in\mathcal{B}_\alpha(\mathcal{S})$ and $T^{**}(x^{**})=x^{**}\restrict{\mathcal{S}}$. 
    
    In order to show that $\mathcal{S}$ is norming for $X_{\omega_1}^{**}$ note that, by the assumption on $\mathcal{S}$, for some universal $c\geqslant 1$,
    \[\|x\|\leqslant c\|T(x)\|, \quad x\in X.\]
    Hence,
    \[\|x^{**}\|\leqslant c\|T^{**}(x^{**})\|,\quad x^{**}\in X^{**}.\]
    For $x^{**}\in X^{**}_{\omega_1}$, $T^{**}(x^{**})\in\mathcal{B}_{\omega_1}(\mathcal{S})$ by the previous step, thus
    \[\|T^{**}(x^{**})\|=\sup\limits_{s\in \mathcal{S}}|T^{**}(x^{**})(s)| = \sup_{s \in \sS} |x^{**}(s)|,\]
    which proves the claim.
    
    To prove the equivalent characterization of weak* convergence it is enough to use the Lebesgue dominated convergence theorem and the fact that $T^{**}$ is a weak* homeomorphic embedding with a weak* closed image (use the Closed Range theorem). 
\end{proof}

\begin{lemma}\label{wuc} Let $X$ be a Banach space. Let $\sum_n x_n^*$ be a weakly unconditionally Cauchy series in $X^*$, weak* convergent to $x^*\in X^*$. Then for every $x^{**}\in X_{\omega_1}^{**}$,
\[x^{**}(x^*)=\sum_{n=1}^\infty x^{**}(x^*_n).\]
\end{lemma}
\begin{proof}
    By the assumption the operator 
    \[T:X\ni x\mapsto(x_n^*(x))_n\in\ell_1\]
    is well-defined and bounded. Its adjoint $T^*:\ell_\infty\longrightarrow X^*$ is given by the formula:
    \[T^*:\ell_\infty\ni (a_n)_n\mapsto\sum_{n=1}^\infty a_nx_n^*\in X^* \]
    where the series on the right-hand side is converging weak*. In particular, for the unit vectors $(e_n)_n\subset\ell_\infty$, 
    \[x^*=T^*\left(\sum_{n=1}^\infty e_n\right).\]
    By induction and weak sequential completeness of $\ell_1$ we obtain  $T^{**}(X_{\omega_1}^{**})\subset\ell_1$. 
    Fix $x^{**}\in X^{**}_{\omega_1}$. By the above, $T^{**}(x^{**})$ is weak* continuous on $\ell_\infty$, therefore 
    \[x^{**}(x^*)=x^{**}\left(T^*\left(\sum_{n=1}^\infty e_n\right)\right)=T^{**}(x^{**})\left(\sum_{n=1}^\infty e_n\right)=\sum_{n=1}^\infty T^{**}(x^{**})(e_n)=\sum_{n=1}^\infty x^{**}(x^*_n).\]
    \end{proof}

\begin{proposition} \label{Prop-Order-Baire1-Sum}
    Let $X_n$, $n \in \N$, be Banach spaces and $X = (\bigoplus X_n)_{\ell_1}$. Then for every $\alpha \leqslant\omega_1$ we have $X^{**}_\alpha \subset (\bigoplus (X_n)^{**}_\alpha)_{\ell_1}\subset X^{**}$. In particular,  $\ord_\bB(X) = \sup_n \ord_\bB(X_n)$ for strictly increasing $(\ord_\bB(X_n))_n$. 
\end{proposition}

\begin{proof}
    Fix $\alpha \leqslant \omega_1$ and pick $y^{**} \in X^{**}_\alpha$. Denote by $\pi_n:X \rightarrow X_n$, $n \in \N$, the canonical projections. We will show that $y^{**} = \sum_{n} \pi_n^{**} y^{**}$, with the series converging absolutely. Then $y^{**} \in (\bigoplus (X_n)^{**}_\alpha)_{\ell_1}$ as  $\pi_n^{**} y^{**} \in (X_n)^{**}_\alpha$, $n\in\N$, by weak*-weak* continuity of $\pi_n^{**}$.

    To see that $\sum_{n} \pi_n^{**} y^{**}$ converges absolutely, note that for any $n \in \N$, $\pi_n^{**} y^{**} \in (X_n)^{**}_\alpha\subset X_n^{**}$ and hence we can find $x_n^* \in X_n^{*}$ with $\norm{x_n^*} = 1$ and $\pi_n^{**} y^{**} (x_n^*) > \norm{\pi_n^{**} y^{**}} - 2^{-n}$. As $X^* = (\bigoplus X_n^*)_{\ell_\infty}$, any functional of the form $y_N^* = (x_1^*,\dots,x_N^*,0,\dots) \in X^*$, $N\in\N$, satisfies $\norm{y_N^*} = 1$. Thus
    \begin{align*}
        \sum_{n=1}^N \norm{\pi_n^{**} y^{**}} < \sum_{n=1}^N \pi_n^{**} y^{**}(x_n^*) + 2^{-n} < 1 + y^{**}(y_N^*) \leqslant 1 + \norm{y^{**}}.
    \end{align*}
    Hence, $\sum_{n} \pi_n^{**} y^{**}$ converges absolutely as required.

    What is left to show is that $\sum_{n} \pi_n^{**} y^{**}$ converges to $y^{**}$. Note that for any $x^* \in X^*$ we have that $x^* = \sum_n \pi_n^* x^*$, where the series on the right hand side is weakly unconditionally Cauchy. Indeed, for any $x \in X$ we have $\sum_n \norm{\pi_n x} = \norm{x}$ by the definition of the norm in $X$, hence
    \begin{align*}
        \sum_{n=1}^\infty |\pi_n^* x^* (x)| \leqslant \sum_{n=1}^\infty \norm{x^*} \norm{\pi_nx} \leqslant \norm{x^*} \norm{x}.
    \end{align*}
    Therefore, by Lemma \ref{wuc} we get $y^{**}(x^*) = \sum_n y^{**}(\pi_n^* x^*) = \sum_n \pi_n^{**} y^{**}(x^*)$ for any $x^*\in X^*$, which proves the required inclusion.

Now assume $(\ord_\bB(X_n))_n$ strictly increases. By the inclusion proved above, any $x^{**}\in X^{**}_{\omega_1}$ is a limit in norm of the sequence $((P_1^{**}+\dots+P_n^{**})x^{**})_n$, thus $\ord_\bB(X)\leqslant\sup_n\ord_\bB(X_n)$.
On other hand, by weak*-weak* continuity of each $P_n^{**}$, $\ord_\bB(X)\geqslant\ord_\bB(X_n)$ for all $n\in\N$, which ends the proof. 
\end{proof}

\begin{remark}
    A slight modification of the proof above shows that the equality $\operatorname{ord}_{\bB}(X) = \sup_n \operatorname{ord}_{\bB}(X_n)$ from Proposition \ref{Prop-Order-Baire1-Sum} holds true provided that $\sup_n\ord_\bB(X_n)$ is not attained in infinitely many $n \in \N$. If the supremum is attained in infinitely many $n \in \N$, the last paragraph of the proof of Proposition \ref{Prop-Order-Baire1-Sum} gives us only that $\ord_\bB(X) \leqslant \sup_n \ord_\bB(X_n) + 1$. The equality is also valid under the following condition: There is a constant $C = C(\alpha) > 0$ such that $\bigcup_{\beta < \alpha} (X_n)^{**}_\beta \cap B_{X^{**}}$ is weak$^*$ sequentially dense in $(X_n)^{**}_\alpha \cap C B_{X^{**}}$ for every $n \in \N$, where $\alpha = \sup_n \ord_\bB (X_n)$. The constant corresponding to $\alpha = 1$ is $C(1) = 1$ for any choice of Banach spaces $(X_n)_{n}$ (\cite[Theorem II.1.2(a)]{ArgyrosGodefroyRosenthal2003}). We do not know if $C(\alpha)$ exists for $\alpha > 1$.
\end{remark}

\subsection{Lindenstrauss spaces}\label{section_lindenstrauss}

Recall that a Banach space $X$ is a \textbf{Lindenstrauss space} if the dual $X^*$ is isometric to $L_1(\mu)$ for some measure $\mu$.

\begin{theorem}
Let $X$ be an infinite dimensional separable real Lindenstrauss space. Then $\ord_\bB (X) \in \{1,\omega_1\}$.
\end{theorem}

\begin{proof} We consider two cases. 

\textbf{Case 1}. The set $\operatorname{ext} B_{X^*}$ of extreme points of its dual ball is countable. Then for every $x^{**} \in X^{**}_{\omega_1}$, the restriction $x^{**}\restrict{\operatorname{ext} B_{X^*}}$ is a bounded Baire-1 function. It follows from \cite[Theorem 2.7(b)]{LudvikSpurny2014} (as every countable set is clearly $F_\sigma$), that $x^{**}\restrict{\operatorname{ext}B_{X^*}}$ can be extended to an element of $X^{**}_1$, and by \cite[Proposition 2.3]{LudvikSpurny2014}, this extension is necessarily $x^{**}$. Hence, $x^{**} \in X^{**}_1$ and $\ord_\bB (X) \leqslant 1$. Alternatively, countability of $\operatorname{ext} B_{X^*}$ implies that $X^*$ is separable (see \cite[Theorem 3.122]{FabianEtAl}), and thus $X^{**} = X^{**}_1$ by the Goldstine's theorem and metrizability of $(B_{X^{**}},w^*)$. As no infinite dimensional Lindestrauss space is reflexive, $X \subsetneq X_1^{**} = X^{**}$ and thus $\ord_\mathcal{B}(X)=1$.
  
\textbf{Case 2}. The set $\operatorname{ext} B_{X^*}$ is uncountable. We claim that for any $\alpha < \omega_1$ there is an odd function $h \in \bB_\alpha(\operatorname{ext} B_{X^*}) \setminus \bigcup_{\beta<\alpha} \bB_\beta(\operatorname{ext} B_{X^*})$. Then by \cite[Theorem 2.7(a)]{LudvikSpurny2014}, $h$ can be extended to some $x^{**} \in X^{**}_{1+\alpha}$, which  clearly cannot be an element of $\bigcup_{\beta < \alpha} X^{**}_\beta$, as otherwise its restriction $h$ would be in $\bigcup_{\beta<\alpha} \bB_\beta(\operatorname{ext} B_{X^*})$. Hence $\ord_\bB(X) \geqslant \alpha$ for any  $\alpha < \omega_1$, thus  $\ord_\bB(X) = \omega_1$.

We prove now that,  
given $\alpha < \omega_1$, there is an odd function $h \in \bB_\alpha(\operatorname{ext} B_{X^*}) \setminus \bigcup_{\beta<\alpha} \bB_\beta(\operatorname{ext} B_{X^*})$.  
We can suppose that $\alpha > 0$. Note first that $\operatorname{ext} B_{X^*}$ is a $G_\delta$-subset of the Polish space $B_{X^*}$ as $X$ is separable (see \cite[Corollary I.4.4]{Alfsen}), and thus $\operatorname{ext} B_{X^*}$ is itself Polish. Hence, $\operatorname{ext} B_{X^*}$ contains a perfect subset $S$. Further, $S \cup -S$ is clearly also perfect. By Zorn's lemma there is a maximal open set $U \subset S\cup-S$ with $U \cap -U = \varnothing$. 

We claim that $U \cup -U$ is dense in $S\cup -S$. Assume otherwise, then there is $x^* \in S \cup -S$ and an open absolutely convex neighborhood $V$ of $0$, such that $x^* \notin V$ and $(x^*+V)\cap(\overline{U}\cup-\overline{U})=\varnothing$. Put $U_1:=U\cup(x^*+V)\cap(S\cup-S)$. We will show that $U_1 \cap -U_1 = \varnothing$. Suppose, towards a contradiction, that there exists $y^* \in U_1 \cap -U_1$. As $U \cap -U = \varnothing$, we can without loss of generality assume that $y^* \notin U$. Then $y^{*}\in (x^*+V)\cap(S\cup-S)$, and hence $y^{*}\notin U\cup-\overline{U}$. It follows that $y^*\notin -U$, which implies $y^*\in -x^*-V=-x^*+V$, and consequently $0\in 2x^*+V+V=2x^*+2V$, which contradicts  the fact that $x^* \notin V$. Thus $U_1\cap -U_1=\varnothing$ which in turn contradicts maximality of $U$ as $x^*\in U_1$ and $x^*\notin U$. 

Observe that $U$ is uncountable  (otherwise $U\cup-U$ would be a countable dense open subset of a perfect space -- a contradiction). As an open subset of Polish space, $U$ is Polish as well, and hence there is $g\in\mathcal{B}_\alpha(U)\setminus\bigcup_{\beta<\alpha}\mathcal{B}_\beta(U)$. Now we can define $h$ as follows.
  \[ h(x^*):=\begin{cases} 
      0 & x^*\in\operatorname{ext}B_X^*\setminus(U\cup-U) \\
      g(x^*) & x^*\in U \\
      -g(-x^*) & x^*\in -U.  
   \end{cases}
  \]
  Then $h$ is clearly odd, $h \in \bB_\alpha( \operatorname{ext} B_{X^*})$ and $h \notin \bigcup_{\beta < \alpha} \bB_{\beta} (\operatorname{ext} B_{X^*})$, as its restriction $g$ satisfies $g \notin \bigcup_{\beta < \alpha} \bB_{\beta} (U)$.  
\end{proof}

\section{The proof of Theorem \ref{Theorem-Construction}}

The section is devoted to the construction and analysis of the space $Y$ of Theorem \ref{Theorem-Construction}. We start with the Azimi-Hagler space, a "building brick" for the construction of the space $Y$.

By an \textbf{interval} in $\N$ we mean an intersection of any interval in $\R$ (either finite or infinite) with $\N$. We call a sequence (finite or infinite) of intervals $(I_n)_n$ \textbf{successive}, provided $\max I_n<\min I_{n+1}$ for any $n \in \N$.

\subsection{Azimi-Hagler space}\label{section_AH_space}

In this subsection we recall the construction of the Azimi-Hagler space \cite{AzimiHagler1986} and some of its properties.

\begin{definition} \label{Definition-AHSpace} \;
    \begin{enumerate}[(i)]
        \item If $I$ is an interval and $x \in c_{00}$, we set $S_I(x) = \sum_{i \in I} x(i)$.
        \item For $x \in c_{00}$ we define 
            \begin{align*}
                \norm{x}_{AH} = \sup \left\{ \sum_{j=1}^n \tfrac{1}{j} |S_{I_j}(x)|: \; (I_1,\dots,I_n) \text{ successive sequence of intervals} \right\}.
            \end{align*}
    \end{enumerate}
We define $X_{AH}$ to be the completion of $c_{00}$ under this norm and call $X_{AH}$ the \textbf{Azimi-Hagler space}.
\end{definition}

We note that in \cite{AzimiHagler1986}, the space $X_{AH}$ was constructed using a sequence $(\alpha_j)_{j}$ of non-negative numbers that satisfy the following properties:
\begin{enumerate}
    \item $\alpha_1 = 1$, $(\alpha_j)_{j}$ is nonincreasing;
    \item $\lim_{j} \alpha_j = 0$;
    \item $\sum_{j} \alpha_j = \infty$;
\end{enumerate}
as weighting factors in the definition of the norm in (ii) of Definition \ref{Definition-AHSpace}. For the purposes of the paper, we consider only the sequence $(\alpha_j)_{j} = (\frac{1}{j})_{j}$, but the results remain true for any other sequence $(\alpha_j)_{j}$ that satisfies the points above.

Note that the norm on $X_{AH}$ is a jamesification (see \cite{BellenotHaydonOdell89}) of a Garling space \cite{Garling1968} with weights $w = (\tfrac{1}{j})_{j}$ (see \cite{AlbiacEtAl2018} for an overview of Garling spaces). 

In the next theorem we list some known properties of the Azimi-Hagler space that  will be needed in the construction of the space $Y$ in the proof of Theorem \ref{Theorem-Construction}.

\begin{theorem} [\protect{\cite[Theorem 1]{AzimiHagler1986}}] \label{Theorem-AHProperties}
    The following is true:
    \begin{enumerate}
        \item The canonical unit basis $(e_n)_{n}$ is a normalised bimonotone spreading boundedly complete Schauder basis of $X_{AH}$;
        \item The canonical unit basis $(e_n)_{n}$ is a nontrivial weak Cauchy sequence, we denote by $\theta \in (X_{AH})^{**}_1$ its weak* limit.
        \item If $x^{**} \in (X_{AH})^{**}_1$ is a Baire-1 element of $X_{AH}^{**}$, then $x^{**} = x + a \theta$ for some $x \in X_{AH}$ and $a \in \R$. That is, $X_{AH}$ is of codimension 1 in $(X_{AH})^{**}_1$.
        \item $X_{AH}$ is $\ell_1$-saturated.
    \end{enumerate}
\end{theorem}

For more information about the Azimi-Hagler space, see \cite{Andrew1987,AzimiHagler1986,LedariParvaneh2013}. For an analogue of the Azimi-Hagler spaces for $p > 1$, see \cite{Azimi2002,AzimiLedari2006}.

\subsection{The space $Y$ and its properties}\label{section_def_Y} From now on we fix a Banach space $X$ with a basis  $(x_n)_{n=1}^\infty$,  and denote by $(x_n^*)_{n=1}^\infty \subset X^*$ the biorthogonal functionals. Renorming, if necessary, we can assume that $(x_n)_n$ is bimonotone and normalized. 

We aim at an infinite direct sum of isometric copies $Z_n$, $n\in\N$, of the Azimi-Hagler space, which allows to place the basis $(x_n)_n$ on the sequence $(\theta_n)_n$ formed by weak* limits of the canonical bases of $Z_n$'s and preserve the Baire class structure of $X$ inside $Y^{**}$. We define $Y$ as the completion of $c_{00}(\Lambda)$, with $\Lambda$ being the lower triangular part of $\N^2$, endowed with the norm given by its norming set $\mathcal{S}$ of functionals. The structure of Azimi-Hagler space and $X$, respectively, is planted in $Y$ by means of $AH$-functionals and $X$-functionals, defined below. The structure of $X$, present in $Y$ only on suitably chosen finite basic sequences, is inherited on the sequence of weak* limits $\theta_n$'s of the canonical bases of $Z_n$'s. Further, in order to incorporate Baire classes of $X$ into Baire classes of $Y$ of desired order we need to eliminate the possibility of approaching Baire functionals of $X$ planted in $Y^{**}$ by sequences of Baire functionals in $Y^{**}$ other than those coming from the Baire class structure of $X$.  Note that such reduction of Baire classes might happen e.g. in the case of $X$ with an unconditional basis and $Y$ being a direct sum $\big(\bigoplus Z_n\big)_X$. To avoid such situation, we close the norming set of the norm of $Y$ under $c_0$-sums of  functionals with supports forming a diagonal sequence (see Definition \ref{Def_Y}), which in particular makes any seminormalized sequence of vectors with diagonal sequence of supports in $Y$ equivalent to the unit vector basis of $\ell_1$ and allows $X$ to be only finitely represented in $Y$. The full description of Baire classes of $Y$ (see Sections \ref{section_baire_1_Y}, \ref{section_baire_Y}) relies heavily on the fact that $Y$ is a dual space (thanks to Azimi-Hagler space being a dual space) combined with general facts of Section \ref{Section_general} applied to the norming set $\mathcal{S}$ of $Y$ defined below. The  observations critical for showing that $Y$ meets the requirements in Theorem \ref{Theorem-Construction} are contained in the key Lemmas \ref{Lemma-wuc-Y}, \ref{property of Y_1**} and \ref{Lemma-Dixmier-weaklyCauchy}.

Recall that every functional $x^*\in X^*$ is uniquely represented as a weak* convergent series $x^*=\sum_{n=1}^\infty x^*(x_n)x_n^*$. Let
\[K:=\Big\{(x^*(x_n))_{n=1}^\infty\in\R^\N : x^*\in B_{X^*}\Big\}.\]
%\left\{(a_n)_n\in\mathbb{R}^\N:\sum_{n=1}^\infty a_nx_n^*\in B_{X^*}, \text{ with the series weak* convergent}\right\}.

Let $\Lambda=\{(m,n)\in\N^2: m\geqslant n\}$. Given $F\subset \N^2$ denote by $\chi_F: \N^2 \to \{0,1\}$ its characteristic function. For any functional $s: c_{00}(\Lambda)\to \R$, we set $\chi_F\cdot s: c_{00}(\Lambda) \rightarrow \R$ the linear extension of the map $e_{m,n} \mapsto \chi_F(m,n)s(e_{m,n})$, $(m,n) \in \Lambda$. Then $\chi_f \cdot s$ is also a functional on $c_{00}(\Lambda)$ and its support is contained in $F \cap \Lambda$.

We call $F\subset\Lambda$ a \textbf{$\Lambda$-rectangle} if $F= (I\times J)\cap \Lambda$ for some (finite or infinite) intervals $I,J\subset \N$, that is, if $F$ is an intersection of a rectangle in $\N^2$ with $\Lambda$. We shall denote the unit vectors in $c_{00}(\Lambda)$ by $(e_{m,n})_{(m,n) \in \Lambda}$. Given a functional $s: c_{00}(\Lambda)\to \R$ we define the \textbf{support} of $s$ by the formula $\supp s=\{(m,n)\in\Lambda: s(e_{m,n})\neq 0\}$ and the \textbf{range} of $s$ ($\range s$) as the smallest $\Lambda$-rectangle containing $\supp s$. 

We say that a sequence of $(F_n)_n$ of subsets of $\Lambda$ is 
\textbf{diagonal}, provided for some successive sequences of intervals $(I_n)_n$, $(J_n)_n$ in $\N$ we have $F_n\subset I_n\times J_n$ for each $n \in \N$.

\begin{notation}
Given any $F\subset\N^2$ we define the \textbf{summing functional} $S_F$ over $F$ 
 on $c_{00}(\Lambda)$ as
\[S_F: c_{00}(\Lambda)\ni x=\sum_{(m,n)\in\Lambda}x_{m,n}e_{m,n}\mapsto \sum_{(m,n)\in F\cap\Lambda}x_{m,n}\in\R.\]
Given a sequence $\mathcal{F}=(F_j)_j$ of pairwise disjoint subsets of $\N^2$ and $\mathbf{a}=(a_j)_j\subset \R$ define a functional on $c_{00}(\Lambda)$ of the form
\[S_\mathcal{F}^\mathbf{a}: c_{00}(\Lambda)\ni x\mapsto \sum_{j=1}^\infty a_jS_{F_j}(x)\in\R.\]
Any functional $S_\mathcal{F}^\mathbf{a}$ with $\mathcal{F}=(I\times \{j\})_j$ for some  interval $I\subset\N$ and $\mathbf{a}\in K$, is called an \textbf{$X$-functional}. In case of infinite $I$, we call such $S_\mathcal{F}^\mathbf{a}$ a \textbf{$SIX$-functional} (a segment-infinite $X$-functional). 

Any functional $S_\mathcal{F}^\mathbf{a}$ with $\mathcal{F}=(I_j\times \{k\})_j$ for some successive intervals $I_1<I_2<\dots$ in $\N$,  $k\in\N$, and $\mathbf{a}=(\tfrac{\varepsilon_j}{j+j_0})_j$ with each $\varepsilon_j=\pm 1$ and $j_0\in\N\cup\{0\}$, is called an \textbf{$AH$-functional}.
\end{notation}

\begin{definition}{\label{Def_Y}}    
Let $\mathcal{S}$ be the set of all functionals $s$ of the form $s=\sum_ns_n$, where each $s_n$ is either an $AH$-functional or an $X$-functional, and the sequence of $\Lambda$-rectangles  $(\range s_n)_n$ is diagonal. 

The set $\mathcal{S}$ defines a norm $\|\cdot\|$ on $c_{00}(\Lambda)$ as its norming set, i.e. 
\[\|y\|:=\sup_{s\in\mathcal{S}}s(y), \text{ for any } y\in c_{00}(\Lambda).\]
We define the space $Y$ as the completion of $(c_{00}(\Lambda), \|\cdot\|)$. 
\end{definition}
We gather properties of the norming set $\mathcal{S}$ in the following lemma.
\begin{lemma}\label{structure of S}
\begin{enumerate}
    \item For any $(s_n)_n\subset \mathcal{S}$ with $(\range s_n)_n$ a diagonal sequence, also $\sum_ns_n\in \mathcal{S}$.    
    \item For any diagonal sequence of $\Lambda$-rectangles $\mathcal{F}=(F_n)_n$ and $s\in\mathcal{S}$, also $\chi_F\cdot s\in\mathcal{S}$, where $F=\bigcup\mathcal{F}$.
    \item Fix $s\in\mathcal{S}$. If $\supp s$  contains an infinite interval (i.e. a set of the form $I\times\{n_0\}$ with $I\subset\N $ unbounded interval), then $s=s_0+s_\infty$ with finitely supported $s_0\in\mathcal{S}$ and an $SIX$-functional $s_\infty$. Otherwise $s$ is a weak* limit of a weakly unconditionally Cauchy series of finitely supported elements of $\mathcal{S}$.
    \item The set $\mathcal{S}\subset Y^*$ is weak* compact.
\end{enumerate}
\end{lemma}
\begin{proof} (1) follows immediately by the definition of $\mathcal{S}$. 

For (2) take any functional $s=\sum_ns_n\in\mathcal{S}$ and a diagonal sequence of $\Lambda$-rectangles $\mathcal{F}=(F_n)_n$. Note that by bimonotonicity of the basis $(x_n)_n$ and definition of $AH$-functionals and $X$-functionals,  $\chi_{F_m} \cdot s_n\in \mathcal{S}$ for any $n,m\in\N$. As the family $(F_m\cap \range s_n)_{n,m}$ can be ordered to be a diagonal sequence in $\Lambda$, $\chi_F\cdot s\in \mathcal{S}$ by (1).

To show (3) take $s\in\mathcal{S}$. By definition, there is a sequence $(s_n)_n$ whose elements are either $AH$-functionals or $X$-functionals, such that $s=\sum_n s_n$ and $(\range s_n)_n$ is diagonal. Note that if some $s_{n_0}$ has infinite support, then by the definition and diagonality of $(\range s_n)_n$, $s=\sum_{n=1}^{n_0}s_n$ with $\supp s_n$ finite for each $n<n_0$. 

First we  deal with the case of $\supp s$ containing an infinite interval. Then, by the above, $s$ is of the form $s=\sum_{n=1}^{n_0}s_n$, with $\supp s_{n_0}$ containing an infinite interval and $\supp s_n$ finite for any $n<n_0$. If $s_{n_0}$ is a $SIX$-functional, the case is settled, otherwise $s_{n_0}$ is an $AH$-functional, i.e.  $s_{n_0}=S^\mathbf{a}_{\mathcal{F}}$ with $\mathcal{F}=(I_j\times \{k\})_j$ for some successive intervals $I_1,I_2,\dots$ in $\N$,  $k\in\N$, and $\mathbf{a}=(a_j)_j\subset [-1,1]$. Observe that in this case there is $m_0 \in \N$ such that $\mathcal{F} = (I_j \times \{k\})_{j=1}^{m_0}$, with  the last interval $I_{m_0}$ infinite and $I_j$, $j<m_0$,  finite. Hence, the $SIX$-functional $s_\infty=a_{m_0}S_{I_{m_0}\times\{k\}}$ and the finitely supported $s_0 = s - s_\infty$ satisfy the conclusion of the lemma. 

Now we deal with the case of $\supp s$ containing no infinite interval. If for every $n \in \N$, $s_n$ has finite support, the series $\sum_n s_n$ is weakly unconditionally Cauchy as $\big\{\sum_n\varepsilon_ns_n:(\varepsilon_n)_n\subset\{-1,1\}^\mathbb{N}\big\}\subset\mathcal{S}$ (see Subsection \ref{basics}). Otherwise, by the above, $s$ is of the form $s=\sum_{n=1}^{n_0}s_n$, with $\supp s_{n_0}$ infinite and $\supp s_n$ finite for any $n<n_0$. As $\supp s$ contains no infinite interval, and $X$-functionals with infinite support are precisely $SIX$-functionals by the definition of $\Lambda$, $s_{n_0}$ is an $AH$-functional. Write $s=\sum_{j=1}^\infty a_jS_{I_j\times \{k\}}$ for suitable $(a_j)_j$, $(I_j)_j$ and $k$, and note the series is weakly unconditionally Cauchy as $\big\{\sum_j\varepsilon_ja_jS_{I_j\times \{k\}}:(\varepsilon_j)_j\subset\{-1,1\}^\mathbb{N}\big\}\subset\mathcal{S}$ (see Subection \ref{basics}), which ends the proof of (3). 

%for some $n \in \N$ we have that $s_n$ has infinite support. In this case, $s_n$ is an $AH$-functional as every $X$-functional with infinite support is a $SIX$-functional by the definition of $\Lambda$, but $SIX$-functionals contain an infinite interval. Then $s_n$ itself is a weak* limit of a weakly unconditionally Cauchy series of finitely supported elements from $\mathcal{S}$, and $(s_k)_k\setminus\{s_n\}$ is necessarily finite and consists of finitely supported functionals by diagonality.

Concerning (4), note first that as $Y$ is separable and $\mathcal{S}$ is bounded, it is enough to prove that the pointwise limit $s$ of a sequence $(s_j)_j\subset\mathcal{S}$ is an element of $\mathcal{S}$. By the definition of $\mathcal{S}$ and compactness of $K$ it suffices to check that for any finite $F=(I\times I)\cap \Lambda\subset\Lambda$, with $I\subset \N$ interval, $\chi_F\cdot s\in \mathcal{S}$. Fix $F=(I\times I) \cap \Lambda$, with $I\subset \N$ interval. By the definition of $\mathcal{S}$  each $\chi_F\cdot s_j$ is of the form $\chi_F\cdot s_j=\sum_n \chi_{F_{j,n}}\cdot s_j$ for some finite diagonal sequence of $\Lambda$-rectangles $(F_{j,n})_n$ with each $\chi_{F_{j,n}}\cdot s_j$  either an $AH$-functional or an $X$-functional. Passing to a subsequence of $(s_j)_j$ we can assume that $(\bigcup_nF_{j,n})_j$ is constant, i.e. $F_{j,n}=F_n$ for all $j$ and a (finite) diagonal sequence of $\Lambda$-rectangles $(F_n)_n$. Again passing to a subsequence  we can assume that for each $n$ separately, either $(\chi_{F_n}\cdot s_j)_j$  consists of $AH$-functionals or $(\chi_{F_n}\cdot s_j)_j$ consists of $X$-functionals. By the definition of $AH$-functionals and the set $K$, for any $n$, the functional $\chi_{F_n}\cdot s$ is either an $AH$-functional, a zero functional, or an $X$-functional, which proves that $\chi_F\cdot s\in\mathcal{S}$.
\end{proof}

\begin{notation}
By Lemma \ref{structure of S}(2), for any $F \subset \N^2$ such that $F \cap \Lambda$ is a $\Lambda$-rectangle, the projection 
\[P_F:c_{00}(\Lambda)\ni x=\sum_{(m,n)\in\Lambda}x_{m,n}e_{m,n}\mapsto \sum_{(m,n)\in F \cap \Lambda}x_{m,n} e_{m,n}\in c_{00}(\Lambda)\]
is bounded with norm 1 with respect to the norm of $Y$. We denote its extension to $Y$ also by $P_F$. Note that the definition of $P_F$ does depend only on $F \cap \Lambda$ but it is convenient for notation to define it for a general $F \subset \N^2$ such that $F \cap \Lambda$ is a $\Lambda$-rectangle.

\end{notation}
The next two observations on the subspace and basis structure of $Y$ follow by the definition of the norming set $\mathcal{S}$ and properties of the Azimi-Hagler space.
\begin{lemma} \label{lemma_sequences_in_Y}
\begin{enumerate}
    \item Any normalized sequence $(y_j)_j\subset c_{00}(\Lambda)$ with $(\supp y_j)_j$ forming a diagonal sequence in $\Lambda$ is equivalent to the unit vector basis of $\ell_1$. 
    \item For any $n\in\N$, the sequence $(e_{m,n})_{m= n}^\infty$ is 1-equivalent to the canonical basis of the Azimi-Hagler space. 
    \item     The space $Y$ is $\ell_1$-saturated.    
\end{enumerate}
\end{lemma}
\begin{proof}
(1) Take a diagonal sequence of $\Lambda$-rectangles $(F_j)_j$ with $\supp y_j\subset F_j$ and $(s_j)_j\subset \mathcal{S}$ with $s_j(y_j)=1$ for each $j\in\N$. By Lemma \ref{structure of S}(1) also $(\chi_{F_j}\cdot s_j)_j$ is equivalent to the unit vector basis of $c_0$ and is biorthogonal to $(y_j)_j$, which ends the proof. 

(2) Fix $n\in\N$ and let $Z=\spn\{e_{m,n}: m\in\N, m \geq n\}\subset Y$.  Note that for any $s\in\mathcal{S}$, $s\restrict{Z}$ is a restriction to $Z$ of either an $AH$-functional to $Z$ or an $X$-functional of the form $a \cdot S_{I\times\{n\}}$ with $|a|\leq1$ (as $K \subset [-1,1]^\N$). On the other hand, for any successive sequence of intervals $I_1,I_2,\dots $ we have  $\pm S_{I_1\times\{n\}}\pm\tfrac{1}{2}S_{I_2\times\{n\}}\pm\dots\in\mathcal{S}$, which ends the proof.   

(3) Fix a closed infinite dimensional subspace $Z$ of $Y$. Consider two cases. 

\textbf{Case 1.} For any $n$ there is $z_n\in Z$ with $\|z_n\|=1$ and $\|P_{\N\times [n,\infty)}z_n\|\geqslant \tfrac{1}{2}$. Therefore for some diagonal sequence of $\Lambda$-rectangles $(F_n)_n$ and normalized sequence $(z_n)_n\subset Z$, $\|P_{F_n}z_n\|\geqslant \tfrac{1}{3}$. By Lemma \ref{structure of S}(2), the sequence $(z_n)_n$ dominates the sequence $(P_{F_n}z_n)_n$, with the latter equivalent to the unit vector basis of $\ell_1$ by Lemma \ref{lemma_sequences_in_Y}(1). It follows that $(z_n)_n$ is equivalent to the unit vector basis of $\ell_1$.

\textbf{Case 2.} There is $N\in\N$ such that for any $z\in Z$ with $\|z\|=1$, $\|P_Fz\|\geqslant \tfrac{1}{2}$, where $F=\N\times \{1,\dots,N\}$. It follows that $P_F$ is an isomorphic embedding of $Z$ into the space $P_F(Y)\cong \bigoplus_{n=1}^N\clspn\{e_{m,n}:m\in\N , m \geqslant n \}$, i.e. a finite direct sum of isometric copies of $X_{AH}$, which are $\ell_1$-saturated by Theorem \ref{Theorem-AHProperties}(4). As a finite direct sum of copies of $X_{AH}$ is isomorphic to a subspace of an $\ell_1$-direct sum of copies of $X_{AH}$, Lemma \ref{Lemma_ell_1_saturation} ends the proof. 
\end{proof}

\begin{proposition}\label{prop_predual_of_Y} 
\begin{enumerate}
    \item The unit vectors $(e_{m,n})_{(m,n)\in\Lambda}$ ordered lexicographically form a boundedly complete basis of $Y$.
    \item $Y$ is isometric to $V^*$, where $V:=\clspn\{e_{m,n}^*:(m,n)\in\Lambda\}\subset Y^*$.
    \item $\mathcal{S}\cap V\subset Y^*$ is weak*  dense in $\mathcal{S}$.
\end{enumerate}
\end{proposition}
\begin{proof}
(2) follows from (1) (see \cite[Theorem 3.2.10]{AlbiacKalton}), (3) follows by (1) and Lemma \ref{structure of S}(2), as for any $s\in\mathcal{S}$, $(P^*_{[1,n]^2}s)_n\subset \mathcal{S}$ weak* converges to $s$ by (1). 

For (1) first note that the sequence of unit vectors $(e_{m,n})_{(m,n)\in\Lambda}$ ordered lexicographically forms a basis of $Y$ (with the basis constant at most $3$) by Lemma \ref{structure of S}(2). For bounded completeness fix $(a_{m,n})_{(m,n)\in\Lambda}\subset \R$ with partial sums of the formal series $\sum\limits_{(m,n)\in\Lambda}a_{m,n}e_{m,n}$ uniformly bounded. Assume towards contradiction that 
for some $(m_j,n_j)_j,(M_j,N_j)_j\subset\Lambda$, $(m_j,n_j)\leqslant_{lex}(M_j,N_j)<_{lex}(m_{j+1},n_{j+1})$, $j\in\N$, and $\delta>0$ we have  $\|y_j\|>\delta$, $j\in\N$, where
\[y_j=\sum_{(m,n)\in E_j}a_{m,n}e_{m,n}, \ \  E_j:=\{(m,n)\in\Lambda: (m_j,n_j)\leqslant_{lex}(m,n)\leqslant_{lex}(M_j,N_j)\},\ \  j\in\N.\] 
As any interval $\{(m,n)\in\Lambda: (\tilde m,\tilde n)\leqslant_{lex} (m,n)\leqslant_{lex} (\tilde M, \tilde N)\}$  can be represented as a union of at most 3 $\Lambda$-rectangles, we can assume that each $E_j$, $j\in\N$, is a $\Lambda$-rectangle.

For any fixed $n_0\in\N$, the sequence $(\sum\limits_{m=1}^M a_{m,n_0}e_{m,n_0})_M$ is bounded by Lemma \ref{structure of S}(2), thus Lemma \ref{lemma_sequences_in_Y}(2) combined with Theorem \ref{Theorem-AHProperties}(1) yields \[\sum_{m:(m,n_0)\in E_j}a_{m,n_0}e_{m,n_0}\xrightarrow{j\to\infty} 0 \text{ for any fixed }n_0.\]

Thus we can pick inductively $(j_i)_i\subset\N$ and $\Lambda$-rectangles $F_i\subset E_{j_i}$, $i\in\N$ with $(F_i)_i$ diagonal, so that $\|P_{F_i}y_{j_i}\|>\delta/2$ for all $i\in\N$. For each $i\in\N$, by Lemma \ref{structure of S}(2) pick $s_i\in\mathcal{S}$ with $\range s_i\subset F_i$ and $s_i(P_{F_i}y_{j_i})=\|P_{F_i} y_{j_i}\|$. By definition of $\mathcal{S}$ and choice of $(F_i)_i$, $\sum_{i=1}^{i_0}s_i\in\mathcal{S}$ for any $i_0$. Now for any $N\in\N$ pick maximal $i_N\in\N$ with $F_{i_N}\subset [1,N]^2$ and estimate
\[\Big\|\sum_{(m,n)\leqslant_{lex}(N,N)}a_{m,n}e_{m,n}\Big\|\geqslant \sum_{i=1}^{i_N}s_i(P_{F_i}y_{j_i})=\sum_{i=1}^{i_N}\|P_{F_i}y_{j_i}\|>\tfrac{1}{2}i_N\delta\]
which contradicts the choice of  $(a_{m,n})_{(m,n)\in\Lambda}$ as $i_N\xrightarrow{N\to\infty}\infty$. 
\end{proof}

Before proceeding to the last result of this section, profiting from  the list of observations above and critical for the next sections, we introduce a new class of functionals; we call $s\in\mathcal{S}$ a \textbf{$WUC$-functional}, if it is not of the form $s_0+s_\infty$, with finitely supported $s_0\in\mathcal{S}$ and a $SIX$-functional $s_\infty$. The name is justified by Lemma \ref{structure of S}(3), whereas applications of Lemma \ref{wuc} (starting from the lemma below) prove the utility of the class of $WUC$-functionals.

The following lemma forms a crucial tool in the analysis of  Baire classes of $Y$, conducted in next sections. It is based on the structure of the norming set $\mathcal{S}$ of $Y$ and information on $Y$ being a dual space, combined with the general observations on Baire functionals stated in Section \ref{Section_general}.
\begin{lemma} \label{Lemma-wuc-Y}
    \begin{enumerate}
        \item If $y^{**}\in Y_{\omega_1}^{**}\cap V^\perp$ then $y^{**}(s)=0$ for any $WUC$-functional $s\in\mathcal{S}$.
        \item The class of $SIX$-functionals is 1-norming for $Y_{\omega_1}^{**}\cap V^\perp$. 
    \end{enumerate}
\end{lemma}
\begin{proof}
For (1) apply Lemma \ref{structure of S}(3) and Lemma \ref{wuc}, for (2) use (1), Lemma \ref{structure of S}(4) and Lemma \ref{norming-set}.
\end{proof}
%We conclude the section with an observation on the subspace structure of $Y$, a side effect of the use of the Azimi-Hagler space as the underlying space in the construction of $Y$ and $\ell_1$ present on "diagonal" block sequences in $Y$.

\subsection{Baire-1 functionals of $Y$}
\label{section_baire_1_Y}
In this section we examine the structure of $Y_1^{**}$ and show the identification $Y_1^{**}\cap V^\perp\cong X$ (Proposition \ref{Prop_X=Y_1^**}), obtaining the case $\alpha=0$ of Theorem \ref{Theorem-Construction} (Proposition \ref{Prop_X=Y_1^**}), the starting point for the general case presented in the next section. 
Reasoning presented here relies on the key Lemma \ref{Lemma-wuc-Y}, properties of the Azimi-Hagler space and the Baire theorem, more precisely the "point of continuity" argument (see the proof of Lemma \ref{property of Y_1**}).

By Lemma \ref{lemma_sequences_in_Y}(2) and Theorem \ref{Theorem-AHProperties}(2), for every $n\in\N$,  $(e_{m,n})_{m=n}^\infty$ is a non-trivial weak Cauchy sequence, denote by  $\theta_n\in Y^{**}_1$ its weak* limit. \begin{lemma}\label{equivalence_x_n_theta_n}
The sequence $(\theta_n)_{n=1}^\infty\subset Y^{**}_1$ is a basic sequence 1-equivalent to $(x_n)_{n=1}^\infty\subset X$.
\end{lemma}
\begin{proof} Fix scalars $b_1,b_2,\ldots,b_m$. Observe that for any $SIX$-functional  $S^\mathbf{a}_\mathcal{F}$, $\mathbf{a}=(a_i)_i\in K$, 
    \[\sum_{i=1}^mb_i\theta_i(S^\mathbf{a}_{\mathcal{F}})=\sum_{i=1}^ma_ib_i.\]
    On the other hand,  $\Big\|\sum\limits_{i=1}^m b_ix_i\Big\|=\sup\Big\{\sum\limits_{i=1}^m a_ib_i: (a_i)_i\in K\Big\}$, which ends the proof by Lemma \ref{Lemma-wuc-Y}.
\end{proof}
\begin{lemma}\label{property of Y_1**}
    For any $y^{**}\in Y_1^{**}\cap V^\perp$  the following hold true.
    \begin{enumerate}
    \item For every $(a_n)_n\in K$,
        \[y^{**}\Big(\sum_{n=1}^\infty a_nS_{\N\times\{n\}}\Big)=\sum_{n=1}^\infty a_ny^{**}(S_{\N\times\{n\}}).\]
        \item The series \[\sum_{n=1}^\infty y^{**}(S_{\N\times\{n\}})\theta_n \] converges in norm to $y^{**}$.
\end{enumerate}
\end{lemma}

\begin{proof} Fix $y^{**}\in Y_1^{**}\cap V^\perp$. 

For $(1)$ fix $(a_n)_n\in K$ and note that   $\sum_{n=1}^\infty a_nS_{\mathbb{N}\times\{n\}}$ is an $X$-functional defined by $(a_n)_n$ and a family $\mathcal{F}=(\N\times\{n\})_n$. Suppose towards contradiction that the sequence $\big(\sum_{n=1}^k a_ny^{**}(S_{\mathbb{N}\times\{n\}})\big)_k$ does not converge to $y^{**}\left(\sum_{n=1}^\infty a_nS_{\mathbb{N}\times\{n\}}\right)$. Then there are $(k_i)_i\subset\mathbb{N}$ and $\delta>0$ with
    \[\left|y^{**}\left(\sum_{n=k_i}^\infty a_nS_{\mathbb{N}\times\{n\}} \right)\right|>\delta, \quad i\in\mathbb{N}.\]
     Let $s_i$ be an $X$-functional of the form  $s_i=\sum_{n=k_i}^\infty a_nS_{\mathbb{N}\times\{n\}}$, $i\in\N$. Note that $\supp s_i\cap \Lambda\subset [k_i,\infty)^2$ for all $i\in\N$, thus  
    \[H:=\{s+s_i: s \in \sS, \supp{s}\subset [0,k_i)^2,\,i\in\mathbb{N}\}\subset \mathcal{S}.\]
    Observe that both $H$ and $\mathcal{S}\cap V$ are weak* dense in $\mathcal{S}$ by Proposition \ref{prop_predual_of_Y}(3); for density of $H$ note that for any  $s\in\mathcal{S}$ and $N\in\mathbb{N}$ there is $\tilde s\in H$ with $P_{[1,N]^2}(\tilde s)=P_{[1,N]^2}(s)$. As $y^{**}\in V^\perp$, for any $s\in H$ we have $|y^{**}(s)|>\delta$, whereas $y^{**}\restrict{\mathcal{S}\cap V} = 0$. It follows that $y^{**}$ has no point of continuity in $\mathcal{S}$ which contradicts the fact that $y^{**}$ is a Baire-1 function on $\mathcal{S}$ (Baire theorem).
    
    For $(2)$ we firstly show that the series $\sum_{n=1}^\infty y^{**}(S_{\mathbb{N}\times\{n\}})\theta_n$ converges in norm. Assume this is not the case, then for some $(k_i)_i\subset\mathbb{N}$ and $\delta>0$ we have
    \[\left\|\sum_{n=k_{2i-1}}^{k_{2i}}y^{**}(S_{\mathbb{N}\times\{n\}})\theta_n\right\|>\delta,\quad i\in\mathbb{N}.\]
    By Lemma \ref{equivalence_x_n_theta_n},
    \[\left\|\sum_{n=k_{2i-1}}^{k_{2i}}y^{**}(S_{\mathbb{N}\times\{n\}})\theta_n\right\|=\left\|\sum_{n=k_{2i-1}}^{k_{2i}}y^{**}(S_{\mathbb{N}\times\{n\}})x_n\right\|.\]
    For any $i\in\N$ pick $(a_n^{(i)})_n\in K$ with
    \[\left\|\sum_{n=k_{2i-1}}^{k_{2i}}y^{**}(S_{\mathbb{N}\times\{n\}})x_n\right\|=\sum_{n=k_{2i-1}}^{k_{2i}}a^{(i)}_ny^{**}(S_{\mathbb{N}\times\{n\}}).\]
    Let $s_i=\sum_{n=k_{2i-1}}^{k_{2i}}a^{(i)}_nS_{\mathbb{N}\times\{n\}}$ and note that $s_i$ is an $X$-functional thanks to bimonotonicity of $(x_n)_n$. Then 
    \[y^{**}(s_i)=\sum_{n=k_{2i-1}}^{k_{2i}}a^{(i)}_ny^{**}(S_{\mathbb{N}\times\{n\}}) =\left\|\sum_{n=k_{2i-1}}^{k_{2i}}y^{**}(S_{\mathbb{N}\times\{n\}})\theta_n\right\|>\delta, \quad i\in\N.\]
    Thus we can repeat the proof from (1) concluding that $\sum_{n=1}^\infty y^{**}(S_{\mathbb{N}\times\{n\}})\theta_n$ converges in norm.
    
    To finish the proof of (2) we apply  Lemma \ref{Lemma-wuc-Y}(2). As both $y^{**}$ and $\sum_{n=1}^\infty y^{**}(S_{\mathbb{N}\times\{n\}})\theta_n$ are null on $V$, it is enough to check their action on any  $SIX$-functional of the form $S^\mathbf{a}_\mathcal{F}$, $\mathbf{a}\in K$, with  $\mathcal{F}=(\mathbb{N}\times\{n\})_n$. We have
    \[\left(\sum_{n=1}^\infty y^{**}(S_{\mathbb{N}\times\{n\}})\theta_n\right)(S_{\mathcal{F}}^\mathbf{a})=\sum_{n=1}^\infty y^{**}(S_{\mathbb{N}\times\{n\}})\theta_n(S_{\mathcal{F}}^\mathbf{a})=\sum_{n=1}^\infty a_ny^{**}(S_{\mathbb{N}\times\{n\}})=y^{**}(S_{\mathcal{F}}^\mathbf{a})\]
    with the last equality guaranteed by (1).
\end{proof}

Lemmas above yield the following isometric identification of $X$ with $Y^{**}_1\cap V^\perp$.
\begin{proposition}\label{Prop_X=Y_1^**}
  The operator
    \[T:X\ni x\mapsto\sum_{n=1}^\infty x_n^{*}(x)\theta_n\in Y^{**}\]
    is a well-defined isometric isomorphism from $X$ onto $Y^{**}_1\cap V^\perp$. 
%    \item \sout{The operator $R:=T^*\circ j:Y^*\to X^*$, where $j:Y^*\hookrightarrow Y^{***}$ is the canonical embedding, is a surjection onto $X^*$.}
\end{proposition}

\begin{proof} Observe that the operator $T$ is a well-defined isometry by Lemma \ref{equivalence_x_n_theta_n}. $T(X)\subset Y^{**}_1\cap V^\perp$ by definition of $(\theta_n)_n$ and the fact that $Y^{**}_1$ is closed in norm. The opposite inclusion is guaranteed by Lemma \ref{property of Y_1**}(2). 
\end{proof}

\subsection{Baire functionals of $Y$}
\label{section_baire_Y}
The section is devoted to the study of general Baire functionals of $Y$, and concluded with the proof of Theorem \ref{Theorem-Construction}. 
We start with identifying the candidate for the operator $L$ in Theorem \ref{Theorem-Construction}. 
%Curiously enough, such candidate is provided by the double adjoint operator of the embedding $X\hookrightarrow Y^{**}$ from Proposition \ref{Prop_X=Y_1^**}, as the lemma below shows. 

\begin{lemma}\label{Lemma-operator-L}
   The operator $\tilde P\circ T^{**}: X^{**}\to Y^{**}$, where $\tilde P:Y^{****}\to Y^{**}$ is the Dixmier projection, is both a weak*-weak* homeomorphism onto its weak* closed image and an isomorphic embedding of $X^{**}$ into $Y^{**}$. Moreover, $(\tilde{P}\circ T^{**})(X^{**})$ is complemented in $Y^{**}$. 
\end{lemma}
\begin{proof}

Consider operators  $R:=T^*\circ j:Y^*\to X^*$, where   $j:Y^*\hookrightarrow Y^{***}$ is the canonical embedding and $T: X\to Y^{**}$ is as in Proposition \ref{Prop_X=Y_1^**}, and $J:X^*\to Y^*$ given by the formula
\[J:X^*\ni x^*\mapsto S_{\mathcal{F}}^{(x^*(x_n))_n}\in Y^*,\]
where $\mathcal{F}:=(\mathbb{N}\times\{n\})_n$. The operator $J$ is bounded, as for any $x^*\in B_{X^*}$, $Jx^*$ is an $X$-functional by definition. We claim that $R\circ J=id_{X^*}$. Indeed, note that for any $\mathbf{a}=(a_n)_n\in K$,  
\[T^*(j(S_{\mathcal{F}}^\mathbf{a}))(x_n)=j(S^\mathbf{a}_{\mathcal{F}})(\theta_n)=\theta_n(S_\mathcal{F}^\mathbf{a})=a_n, \quad n\in\N.\]
It follows that $R(S_{\mathcal{F}}^\mathbf{a})=\sum_{n=1}^\infty a_nx_n^*$ for any $\mathbf{a}=(a_n)_n\in K$, thus $R\circ J=id_{X^*}$ as claimed. 

In particular, $R$ is a surjection onto $X^*$. Therefore, $\tilde P\circ T^{**}=R^*$ is a weak*-weak* homeomorphism onto its weak* closed image, as it is injective and weak*-weak* open (onto its image) by surjectivity of $R$, thus also a norm isomorphism onto its image (use the Closed Range theorem).

Moreover, as $R\circ J=id_{X^*}$, $(J\circ R)^*: Y^{**}\to Y^{**}$ is a projection onto $R^{*}(X^{**})=(\tilde P\circ T^{**})(X^{**})$. \end{proof}

It follows easily that the isomorphic embedding $X^{**}\hookrightarrow Y^{**}$ from Lemma \ref{Lemma-operator-L} above, restricted to $X$, coincides with the operator $T$ from Proposition \ref{Prop_X=Y_1^**}, thus carries $X$ to $Y_1^{**}\cap V^\perp$. In order to prove that the embedding in question identifies, moreover, any class of Baire-$\alpha$ functionals of $X$ with the trace of the class of Baire-$(1+\alpha)$ functionals of $Y$ on $V^\perp$, as required in Theorem \ref{Theorem-Construction}, one needs to match weak* topologies on $X^{**}$ and its image in $Y^{**}$. Suitable tools are presented below, based on the key Lemma \ref{Lemma-wuc-Y}, general observations of Section \ref{Section_general} and structure of the norming set $\mathcal{S}$.

Recall that the space $Y$ is the dual space of $V$ (Proposition \ref{prop_predual_of_Y}), denote by $P$ the Dixmier projection  $P:Y^{**}\to Y=V^*$. 

The next lemma makes up for the fact that the operator $\iota\circ P:Y^{**}\to Y^{**}$ is not necessarily weak*-weak* continuous, where $P:Y^{**}\to Y$ is the Dixmier projection and $\iota: Y\hookrightarrow Y^{**}$ is the canonical embedding (whereas obviously $P:Y^{**}\to Y=V^*$ is weak*-weak* continuous) and forms a substantial step in the proof of Theorem \ref{Theorem-Construction}.

\begin{lemma} \label{Lemma-Dixmier-weaklyCauchy} If $(y_n^{**})_{n=1}^\infty\subset Y_{\omega_1}^{**}$ weak* converges then $(Py^{**}_n)_{n=1}^\infty\subset Y$ has a weak Cauchy subsequence. 
\end{lemma}
\begin{proof}
    Put $y_n:=Py_n^{**}$, $n\in\N$. Obviously $y_n^{**}-y_n\in Y^{**}_{\omega_1}\cap V^\perp$, $n\in\N$. Passing to a subsequence and relabeling we can assume that $(S_{I\times\{k\}}(y_n))_n$ converges for every infinite interval $I\subset\N$ and $k\in\N$.
    
    Suppose that $(y_n)_n$ is not weak Cauchy (i.e. not weak* convergent in $Y^{**}$). By Lemma \ref{norming-set} and Lemma \ref{structure of S}(4) there is $s\in\mathcal{S}$,  $(n_l)_l\subset\mathbb{N}$ and $\delta>0$ such that
    \[s(y_{n_{2l}}-y_{n_{2l-1}})>\delta, \quad l\in\N.\]
    Let $z_l:=y_{n_{2l}}-y_{n_{2l-1}}$, $l\in\N$. Observe that $( \chi_{\N\times [1,k]}\cdot s)(z_l)\xrightarrow{l\to\infty} 0$  for every $k\in\mathbb{N}$. Indeed, if $\chi_{\N\times [1,k]}\cdot s$ is a $WUC$-functional, then
    \[(\chi_{\N\times [1,k]}\cdot s)(z_l)=(y^{**}_{n_{2l}}-y^{**}_{n_{2l-1}})(\chi_{\N\times[1,k]}\cdot s)\xrightarrow{l\to\infty}0\]
    where the first equality follows from Lemma \ref{Lemma-wuc-Y}(1), and the limit is zero as $(y_{n}^{**})_n$ is weak* converging. Otherwise, by Lemma \ref{structure of S}(3), $\chi_{\N\times[1,k]}\cdot s=s_0+s_\infty$
    for some $s_0\in\mathcal{S}\cap V$ and  a $SIX$-functional $s_\infty$     with $s_\infty\in\spn\{S_{I\times\{j\}}:1\leqslant j\leqslant k, I\text{ infinite interval in }\N\}$. Hence, as $y_n^{**}-y_n\in Y^{**}_{\omega_1}\cap V^\perp$, $n\in\N$, and as $(s_\infty(y_n))_n$ converges by the assumption, we get
    \[ (\chi_{\N \times [1,k]} \cdot s)(z_l) = (y_{n_{2l}}^{**} - y_{n_{2l-1}}^{**}) (s_0) + s_\infty(z_l) \xrightarrow{l\to\infty} 0.\]    
    
   Hence, as $(e_{m,n})_{(m,n)\in\Lambda}$ is a basis of $Y$, we can pick inductively a diagonal sequence of $\Lambda$-rectangles $(F_i)_i$ and $(l_i)_i\subset\N$ such that $(\chi_{F_i}\cdot s)(z_{l_i})>\delta/2$ for each $i\in\N$ and $(\chi_{F_j}\cdot s)(z_{l_i})<\delta/8^{\min\{i,j\}}$ for all $i\neq j$. Set $D=\bigcup_i F_i$ and note that $\chi_D\cdot s\in\mathcal{S}$ is a $WUC$-functional by Lemma \ref{structure of S}(2), thus $y_n^{**}(\chi_D\cdot s)=(\chi_D\cdot s)(y_n)$, $n\in\N$, by Lemma \ref{Lemma-wuc-Y}. Therefore we arrive at a contradiction with the  weak* convergence of $(y_n^{**})_n$, as $(\chi_D\cdot s)(z_{l_i})>\delta/4$ for each $i$, thus $((\chi_D\cdot s)(y_n))_n$ does not converge by the choice of $D$.  
\end{proof}

\begin{corollary}\label{sequential-density}
$\bigcup_{\beta<\alpha}Y_\beta^{**}\cap V^\perp$ is weak* sequentially dense in $Y_\alpha^{**}\cap V^\perp$ for any $\alpha\leqslant\omega_1$.
\end{corollary}

\begin{proof}
Fix $y^{**}\in Y_\alpha^{**}\cap V^\perp$. By Lemma \ref{Lemma-Dixmier-weaklyCauchy}  there is a sequence $(y_n^{**})_n\subset\bigcup_{\beta<\alpha}Y_{\beta}^{**}$ weak* convergent to $y^{**}$ such that $(Py^{**}_n)_n$ weak* converges in $Y^{**}$. Let $z^{**}$ be the weak* limit of $(Py_n^{**})_n$ and note that $z^{**}\in Y_1^{**}\cap V^\perp$, as $Py_n^{**}(v)=y_n^{**}(v)\xrightarrow{n\to\infty} 0$ for each $v\in V$ by the assumption on $y^{**}$. Then the sequence $(y_n^{**}-Py_n^{**}+z^{**})_n\subset\bigcup_{\beta<\alpha}Y_\beta^{**}\cap V^\perp$ weak* converges to $y^{**}$. 
\end{proof}

\begin{proof}[Proof of Theorem \ref{Theorem-Construction}] For a Banach space $X$ with a normalized bimonotone basis $(x_n)_n$ take the space $Y$ defined in Section \ref{section_def_Y}. The space $Y$ has a  boundedly complete basis by Proposition \ref{prop_predual_of_Y}(1) and is $\ell_1$-saturated by Lemma \ref{lemma_sequences_in_Y}(3). We shall prove that the embedding $L=\tilde P\circ T^{**}: X^{**}\hookrightarrow Y^{**}$, with the operator $T: X\to Y^{**}$ from Proposition \ref{Prop_X=Y_1^**} and Dixmier projection $\tilde P: Y^{****}\to Y^{**}$,  satisfies the assertion of Theorem \ref{Theorem-Construction}. By Lemma \ref{Lemma-operator-L} it is enough to verify
that $Y_{1+\alpha}^{**} = Y\oplus L(X_{\alpha}^{**})$, for any $\alpha\leqslant\omega_1$ (identifying in notation $Y$ and its canonical image in $Y^{**}$). 

Obviously, for any $\alpha\leqslant\omega_1$,
\[Y_{\alpha}^{**}=(P(Y^{**})\cap Y^{**}_\alpha)\oplus (\text{ker} P\cap  Y^{**}_\alpha)=\iota(Y)\oplus (Y_\alpha^{**}\cap V^\perp),\]
therefore we need to check that  $L(X^{**}_\alpha)=Y_{1+\alpha}^{**}\cap V^\perp$ for every $\alpha \leqslant \omega_1$. We show this equality by induction. Firstly, observe that $L\restrict{X}=T$ by the definition of operators involved. Therefore the case $\alpha=0$ is settled by Proposition \ref{Prop_X=Y_1^**}. 
    
Fix now $1\leqslant\alpha\leqslant\omega_1$ and assume the assertion holds for all $\beta<\alpha$. For $L(X^{**}_\alpha)\subset Y_{1+\alpha}^{**}\cap V^\perp$ take any  $x^{**}\in X_{\alpha}^{**}$ and a sequence $(x_n^{**})_n\subset\bigcup_{\beta<\alpha}X_{\beta}^{**}$ weak* convergent to $x^{**}$. By weak*-weak* continuity of $L$ and the inductive assumption, the sequence $(Lx^{**}_n)_n\subset \bigcup_{\beta<\alpha}Y_{1+\beta}^{**}$ weak* converges to $Lx^{**}$, hence $Lx^{**}\in Y_{1+\alpha}^{**}$. 
    
For the opposite inclusion, given $y^{**}\in Y_{1+\alpha}^{**}\cap V^\perp$, pick by Corollary \ref{sequential-density}  a sequence $(y_n^{**})_n\subset\bigcup_{\beta<1+\alpha}Y_{\beta}^{**}\cap V^\perp$ weak* convergent to $y^{**}$. For finite $\alpha$, $(y_n^{**})_n \subset Y_{\alpha}^{**}\cap V^{\perp}$ and thus by the inductive hypothesis $(L^{-1}(y_n^{**}))_n\subset X_{\alpha-1}^{**}$. For infinite $\alpha$,
    \[\bigcup_{\beta<1+\alpha}Y_{\beta}^{**}\cap V^\perp=\bigcup_{\beta<\alpha}Y_{1+\beta}^{**}\cap V^\perp,\]
hence by the inductive hypothesis $(L^{-1}(y^{**}_n))_n\subset\bigcup_{\beta<\alpha}X_\beta^{**}$. In both cases, by Lemma \ref{Lemma-operator-L}, $(L^{-1}(y_n^{**}))_n$ weak* converges to $L^{-1}(y^{**})\in X_\alpha^{**}$.

Finally, we prove that $\ord_\bB(Y) = 1 + \ord_\bB(X)$.
For $\alpha:=\ord_\bB(X)$ we have
     \[Y^{**}_{\omega_1}=Y\oplus L(X^{**}_{\omega_1})=Y\oplus L(X^{**}_\alpha)=Y^{**}_{1+\alpha},\]
     thus $\ord_\bB(Y)\leqslant 1+\alpha$. 
Let $\beta:=\ord_\bB(Y)$ and note that $\beta\geqslant 1$. For finite $\beta$,
\[Y\oplus L(X^{**}_{\beta-1})=Y^{**}_\beta=Y^{**}_{\omega_1}=Y\oplus L(X^{**}_{\omega_1}).\]
As $L$ is injective, $X^{**}_{\omega_1}=X^{**}_{\beta-1}$, thus $1+\ord_\bB(X)\leqslant \beta$. For infinite $\beta$, 
\[Y\oplus L(X^{**}_\beta)=Y^{**}_{1+\beta}=Y^{**}_\beta=Y^{**}_{\omega_1}=Y\oplus L(X^{**}_{\omega_1}),\]
and, as above, $\ord_\bB(X)\leqslant\beta$, thus $1+\ord_\bB(X)\leqslant 1+\beta=\beta$, which ends the proof.
\end{proof}

\section{Main results}
We present here corollaries of Theorem \ref{Theorem-Construction}, providing answers to the questions stated in the introduction. First we provide the following solution of Problem \ref{Problem-NontrivialOrder} for $\alpha\leqslant\omega$. 

\begin{theorem} \label{Theorem-MainOrder}
    For every $\alpha \leqslant\omega$ there is an $\ell_1$-saturated separable Banach space $Y(\alpha)$ of the Baire order $\alpha$. 
\end{theorem}

\begin{proof} For $n\in\N\cup\{0\}$ we recursively construct Banach spaces $Y(n)$ with a basis and of Baire order $n$. For $n=0$ take $Y(0)=\ell_1$. Having defined $Y(n)$ with a basis for a fixed $n\in\N\cup\{0\}$, build $Y(n+1)$ by means of Theorem \ref{Theorem-Construction} applied to $X=Y(n)$.

For $\alpha=\omega$, pick, for $n \in \N$,  Banach spaces $Y(n)$ constructed above. Then $Y(\omega) = (\bigoplus Y(n))_{\ell_1}$ has the Baire order $\omega$ by Proposition \ref{Prop-Order-Baire1-Sum} and is $\ell_1$-saturated by Lemma \ref{Lemma_ell_1_saturation}.
\end{proof}
The next result answers in negative the question in Problem \ref{Problem-Universal}. 
\begin{theorem} \label{Theorem-MainUniversal}
    There is an $\ell_1$-saturated separable Banach space of the Baire order $\omega_1$.
\end{theorem}

\begin{proof}
Apply Theorem \ref{Theorem-Construction} with $X = \mathcal{C}([0,1])$. 
\end{proof}
Finally, we present some analogue of Lindenstrauss' theorem \cite{Lindenstrauss1971} in the Baire classes setting.

\begin{theorem}\label{Theorem-LindenstraussAnalogue}
    For every Banach space $X$ with a basis and $n\in\N$ there is a Banach space $Y$ with a boundedly complete basis such that
    \[Y_n^{**}\cong Y_{n-1}^{**}\oplus X.\]
More precisely, there is an isomorphic  embedding $Q:X\hookrightarrow Y^{**}_n$ such that $Y^{**}_n=Y_{n-1}^{**}\oplus Q(X)$.
\end{theorem}
\begin{proof} We proceed by induction. The case $n=1$ follows by Theorem \ref{Theorem-Construction}. Fix $n\in\N$ and assume  there is a Banach space $Z$  with a basis and an isomorphic embedding $\tilde Q:X\hookrightarrow Z_n^{**}$ with
\[Z_n^{**}=Z_{n-1}^{**}\oplus \tilde Q(X).\]
By Theorem \ref{Theorem-Construction} there is a Banach space $Y$ with a boundedly complete basis and an embedding $L:Z^{**}\hookrightarrow Y^{**}$ such that for every $m\in\N$,
\[Y_m^{**}=\iota(Y)\oplus L(Z_{m-1}^{**}).\]
In particular, applying the above for $m=n+1$ and $m=n$ we obtain
\begin{align*}
Y_{n+1}^{**}&=\iota(Y)\oplus L(Z_n^{**})=\iota(Y)\oplus L(Z_{n-1}^{**}\oplus \tilde Q(X))
=\iota(Y)\oplus L(Z_{n-1}^{**})\oplus L(\tilde Q(X))\\
&=Y_n^{**}\oplus L(\tilde Q(X)).
\end{align*}
Hence $Q:=L\circ \tilde Q:X\hookrightarrow Y_{n+1}^{**}$ is the desired map.
\end{proof}

\begin{remark} 
One can consider the construction of the space $Y$ in Section \ref{section_def_Y}  with the canonical basis of the Azimi-Hagler space replaced by the summing basis of the James space, which would require small modifications in proofs. However, we have chosen the prior due to the form of the functionals, handy in the definition of the norm on $Y$. Further, the choice of the Azimi-Hagler basis makes the resulting space $\ell_1$-saturated, while the choice of the summing basis of James space would make it $\{\ell_1,\ell_2\}$-saturated, in the sense that every closed infinite-dimensional subspace would contain an isomorphic copy of either $\ell_1$ or $\ell_2$. Other variants of the construction of $Y$ will appear in the future article (in preparation) setting the framework for Banach spaces of higher Baire order.   
\end{remark}

\end{document}